\documentclass[10pt]{amsart}
\usepackage{amsmath}
\usepackage{amsfonts}
\usepackage{amsthm}
\usepackage{amssymb}
\usepackage{graphicx}
\newtheorem{theorem}{Theorem}
\newtheorem*{theorem*}{Theorem}
\newtheorem*{acknowledgement*}{Acknowledgement}
\newtheorem*{definition*}{Definition}

\newtheorem{corollary}[theorem]{Corollary}

\newtheorem{lemma}[theorem]{Lemma}

\newtheorem{remark}[theorem]{Remark}

\newcommand{\RR}[0]{\mathbb{R}}
\newcommand{\pd}[2]{\frac{\partial #1}{\partial#2}}

\newcommand{\pdt}[0]{\frac{\partial}{\partial t}}

\newcommand{\gt}[0]{\tilde{g}}

\newcommand{\delb}[0]{\overline{\nabla}}
\newcommand{\delt}[0]{\widetilde{\nabla}}

\newcommand{\Gamt}[0]{\widetilde{\Gamma}}

\newcommand{\Rc}[0]{\operatorname{Rc}}
\newcommand{\Rm}[0]{\operatorname{Rm}}

\newcommand{\Rmt}[0]{\widetilde{\operatorname{Rm}}}

\newcommand{\dfn}[0]{\doteqdot}

\newcommand{\CC}[0]{\mathbb{C}}

\newcommand{\Lc}[0]{\mathcal{L}}
\newcommand{\Ec}[0]{\mathcal{E}}

\newcommand{\Hc}[0]{\mathcal{H}}
\newcommand{\Bc}[0]{\mathcal{B}}

\newcommand{\Kc}[0]{\mathcal{K}}

\newcommand{\nabt}[0]{\tilde{\nabla}}
\newcommand{\Rt}[0]{\tilde{R}}
\newcommand{\Rct}[0]{\widetilde{\Rc}}
\newcommand{\Met}[0]{\operatorname{Met}}

\title[Persistence of bounded curvature under the Ricci flow]{Short-time persistence of bounded curvature under the Ricci flow}

\author{Brett Kotschwar}
\address{Arizona State University, Tempe, AZ, USA}
\email{kotschwar@asu.edu}

\begin{document}
\begin{abstract}
We use a first-order energy quantity to prove a strengthened statement of uniqueness
for the Ricci flow. One consequence of this statement is that if a complete solution on a noncompact manifold has uniformly
bounded Ricci curvature, then its sectional curvature will remain bounded for a short time if it is bounded initially.
In other words, the Weyl curvature tensor of a complete solution to the Ricci flow cannot become unbounded instantaneously  if the Ricci curvature remains bounded.  
\end{abstract}
\maketitle

\section{Introduction}  
Let $M$ be a smooth manifold of dimension $n$ and let $\Met(M, K)$ denote the set of Riemannian metrics $g$ on $M$ for which $(M, g)$
is complete and $\sup|\Rm(g)|_g \leq K$.  The standard statements of existence and uniqueness
for the Ricci flow
\begin{equation}\label{eq:rf}
 \pdt g = -2\Rc(g),
\end{equation}
which are a combination of results of Hamilton \cite{Hamilton3D}, DeTurck \cite{DeTurck}, Shi \cite{Shi}, and Chen-Zhu \cite{ChenZhu}, pertain
to initial metrics and solutions within these classes. They assert that, for
any $g_0\in \Met(M, K_0)$, there are positive constants $K$ and $T$, depending only on $n$ and $K_0$,
and a solution $g(t)$ to the Ricci flow on $M\times [0, T]$ such that $g(0) = g_0$ and $g(t) \in \Met(M, K)$ for each $t\in [0, T]$. 
Moreover, any other solution
$\gt(t)$ satisfying $\gt(0) = g_0$ and $\gt(t)\in \Met(M, \tilde{K})$ for some $\tilde{K}$ on  for $t\in [0, T]$ 
must agree identically with $g(t)$ on $[0, T]$.     

At present, there are essentially no stronger guarantees of short-time existence and uniqueness for \eqref{eq:rf}
which do not require additional geometric or dimensional restrictions.
One might reasonably hope for a more flexible theory of well-posedness, in which one could run the Ricci flow beginning at metrics 
that are incomplete, of unbounded curvature,
or are otherwise singular in some respect, and be assured that the solution would retain any isometries or other special features of the initial metric.
Such a theory would find many geometric applications, and, to this end, extensions to the standard statements of short-time existence and uniqueness
have already been explored in a variety of special cases.

Notably, in dimension two, a program of Topping and Giesen-Topping \cite{GiesenTopping1, GiesenTopping2, Topping1}, 
culminating in the paper \cite{Topping2},
has extended the classical results
of Hamilton \cite{Hamilton2D}, Chow \cite{Chow}, and Shi \cite{Shi} into an essentially complete theory of well-posedness for the Ricci flow
on Riemann surfaces. A consequence of their work is that \emph{any} initial surface (including those that are 
incomplete and without curvature bound) may be flowed in a unique way by a smooth, instantaneously complete, solution to \eqref{eq:rf} on a maximal time interval
of explicitly computable length.

It is unclear how much of Giesen-Topping's theory one can hope to extend to higher dimensions. When $n=2$, 
the Ricci flow is a conformal
flow, whose study can be reduced to that of a single scalar PDE, whereas,
for $n\geq 3$, it is a full system of quasilinear equations.  Furthermore, in every dimension greater than two,
there are complete Riemannian manifolds which are heuristically expected not to admit \emph{any} short-time solution 
to \eqref{eq:rf}; one can imagine, e.g., a noncompact Riemannian manifold with 
an infinite sequence of round cylindrical necks of arbitrarily small caliber (see, e.g., pp. 1-2 of \cite{CabezasRivasWilking} or
p. 11 of \cite{ToppingICM}). 

There are nevertheless a number of existence results for categories of initial metrics not covered by the classical theory. For example, Cabezas-Rivas and Wilking
\cite{CabezasRivasWilking}, building on work of Simon \cite{Simon2, Simon3}, have constructed smooth solutions
to the Ricci flow beginning from arbitrary complete metrics of nonnegative complex (possibly unbounded) sectional curvature.
In related work, Chau, Li, and Tam \cite{ChauLiTam1, ChauLiTam2}, 
and Yang and Zheng \cite{YangZheng} have constructed solutions to the K\"ahler-Ricci flow beginning at metrics of potentially unbounded curvature. Other
recent work includes that of Koch and Lamm \cite{KochLamm} (see also \cite{SchnuererSchulzeSimon} and \cite{Simon1}), who have proven the existence of solutions to the Ricci flow starting at initial data of low-regularity that are
perturbations of the Euclidean metric on $\RR^n$,
Schulze-Simon \cite{SchulzeSimon} and Deruelle \cite{Deruelle} (see also \cite{FeldmanIlmanenKnopf}), who
 have constructed smooth expanding self-similar solutions emerging from conical initial data, and Xu \cite{Xu}, who has proven the short-time existence
of solutions under some integral bounds on the curvature.  

Comparatively less seems to be known about the \emph{uniqueness} of such nonclassical solutions in higher dimensions. To our knowledge, in fact, it is an open question,
except in some special cases, whether (or to what degree) \emph{any} of the constructions referenced in the previous paragraph are unique. There are simple counterexamples 
which demonstrate that
some restriction on the completeness of a solution is necessary to ensure its uniqueness, but it is unclear
what (if any) additional restrictions are actually required. There are a few recent results in this direction, however. 
For example, the local estimates obtained by Chen  in \cite{ChenStrongUniqueness} imply that any two smooth complete solutions on
a three-dimensional manifold which coincide initially at a complete metric of bounded, nonnegative sectional curvature must agree identically (see also \cite{ChenXuZhang}).
Sheng and Wang \cite{ShengWang} have studied complete solutions to the Ricci flow satisfying a lower bound on the complex sectional curvature 
and have proven uniqueness within this class under some additional assumptions on the initial metric. Also, Huang and Tam \cite{HuangTam} have recently obtained 
results concerning the preservation of the K\"ahler structure for solutions beginning at K\"ahler metrics of potentially unbounded curvature, and Chodosh \cite{ChodoshExpanders},
Chodosh-Fong \cite{ChodoshFong}, and Deruelle \cite{DeruelleExpanderUC} have obtained uniqueness results for various classes of
asymptotically conical expanding self-similar solutions.

\subsection{Methods of proof} The uniqueness of complete solutions to the Ricci flow of uniformly bounded curvature 
was proven separately for compact and noncompact $M$.
The statement for compact $M$ was first obtained by Hamilton \cite{Hamilton3D} as a byproduct of the inverse function theorem-based argument
he employed to prove the short-time existence of solutions. Later, in \cite{HamiltonSingularities}, Hamilton gave a simpler proof exploiting a connection
between DeTurck's trick and the harmonic map heat flow. The first proof for noncompact $M$ was given much later by Chen and Zhu \cite{ChenZhu}. 
Their argument is also based on Hamilton's variation on DeTurck's trick, but its implementation in the noncompact setting requires them to overcome a number
of novel technical obstacles.

In \cite{KotschwarRFUniqueness}, we gave an alternative proof of the uniqueness which
circumvents the gauge-related degeneracy of \eqref{eq:rf} without an appeal to DeTurck's trick. There, the problem is instead first reframed
as one for an prolonged system composed of the difference of the solutions, their connections,
and their curvature tensors. This leads to a closed system of differential
inequalities to which a direct $L^2$-energy argument modeled on that for linear parabolic equations can be applied. In a nutshell, the argument is this: given two smooth 
complete solutions $g(t)$ and $\gt(t)$ of bounded curvature on $M\times [0, T]$ with $g(0) = \gt(0)$ 
one can select a weight function $\Phi = \Phi(x, t)$ such that the quantity
\begin{equation}\label{eq:olde}
 \Ec(t) \dfn \int_M\left(t^{-1}|g- \gt|_{\gt}^2 + t^{-1/2}|\Gamma - \Gamt|^2_{\gt} + |\Rm(g) - \Rm(\gt)|^2_{\gt}\right)\Phi\,d\mu_{\gt}
\end{equation}
is well-defined and differentiable on $(0, T]$. A straightforward calculation then shows that $\frac{d}{dt}\Ec(t) \leq C \Ec(t)$ for some $C$
and Gronwall's inequality implies that $\Ec(t) \equiv 0$. With some adjustments, the argument can be extended 
to solutions with some mild growth of curvature in space and some blow-up in time as $t\searrow 0$. This approach
has also found some subsequent application to the uniqueness of other geometric PDE whose degeneracy arises, as in the case of the Ricci flow,
 from the diffeomorphism invariance of the equation (see, e.g.,
\cite{BedulliHeVezzoni}, \cite{Bell}, \cite{Hilaire}, \cite{KotschwarKCFUniqueness}, \cite{LotayWei}).

\subsection{A first-order energy quantity}
One limitation of the approach in \cite{KotschwarRFUniqueness} is that the energy quantity $\Ec(t)$
in \eqref{eq:olde} involves the curvature tensors of the solutions, and factors of both $|\Rm(g)|$ and $|\Rm(\gt)|$ 
crop up in the various error terms in the differential inequality satisfied by 
$\Ec(t)$.
In this paper, we instead apply the strategy in \cite{KotschwarRFUniqueness} to a alternative system encoding a \emph{first-order} prolongation
of the difference of metrics. In place of \eqref{eq:olde}, for each $r > 0$,  we consider a localized energy quantity of the form
\begin{equation}\label{eq:newe}
  \Ec_r(t) \dfn \int_M\left(t^{-(1+\sigma)}|g-\gt|^2_{\gt} +at^{-\sigma}|B|^2_{\gt}\right)\theta_r\Phi\,d\mu_{\gt},
\end{equation}
where $B_k = g_{km}g^{ij}(\Gamma_{ij}^m - \Gamt_{ij}^m)$ is the \emph{Bianchi one-form} (see Section \ref{sec:prelim} below),
$\sigma$ and $a$ are positive constants, $\Phi$ is a rapidly decaying weight function as above, and $\theta_r$ is family of cutoff functions whose supports exhaust $M$ as $r\to \infty$. 
For an appropriate class of solutions, a calculation similar to that in Section 4 of \cite{KotschwarRFUniqueness} implies 
 a uniform bound in $t$ of the form $\Ec_r(t) \leq \epsilon(r)$ where $\lim_{r\to\infty}\epsilon(r) =  0$. 
The novel feature of $\Ec_r(t)$ is that its evolution equation involves the full curvature tensor of only \emph{one} of the solutions.  Moreover, the coefficients in the error
terms that arise depend at most linearly on (any contractions of) the derivatives of curvatures of either of the solutions. In addition to simplifying the proof
of the standard uniqueness theorem for the Ricci flow, these features make possible the following further generalization.

\begin{theorem}\label{thm:uniqueness} Suppose that $g(t)$ and $\gt(t)$ are smooth, complete solutions to the Ricci flow on $M\times [0, T]$ 
satisfying
\begin{equation}\label{eq:curvgrowth}
   \sup_{M\times [0, T]} t^{1-\sigma}|\Rc(g(t))|_{g(t)} \leq K, \quad \sup_{M\times[0, T]}t^{1-\sigma}|\Rm(\gt(t))|_{\gt(t)} \leq K,
\end{equation}
for some $K > 0$ and $\sigma \in (0, 1)$.  If $g(0) = \gt(0)$, then $g(t) = \gt(t)$ for all $t\in [0, T]$.
\end{theorem}

The rate of potential blow-up in $t$ permitted here is an improvement over that in Corollary 3 in \cite{KotschwarRFUniqueness},
in which the curvature tensors
of both solutions were assumed to satisfy bounds of the form $K/t^{1-\sigma}$ for $\sigma\in(1/2, 1)$. 
(We do not treat the case here, however, as in Corollary 4 of that reference, in which
 the curvature tensors of the solutions are permitted to have quadratic growth on each time-slice.) Note that, by the argument of \cite{Hamilton3D}
 (see, e.g., Lemma 6.49  of \cite{ChowKnopf}),
our assumptions on the speeds of $g(t)$ and $\gt(t)$ imply that these metrics remain uniformly equivalent for all $t\in [0, T]$. It would be very interesting to know what can be said about solutions 
whose curvature tensors satisfy a uniform bound of the form $K/t$.

Together
with the short-time existence theorem of Shi \cite{Shi}, Theorem \ref{thm:uniqueness} implies that if the Weyl tensor of any complete solution to the Ricci flow
is bounded initially, it cannot \emph{immediately} become unbounded if the Ricci tensor remains bounded.
\begin{corollary}\label{cor:bdrc}
 Suppose that  $g(t)$ is a solution to the Ricci flow on $M\times [0, T]$ for some $T > 0$. If $g(0) \in \Met(M, K_0)$ and
\[
\sup_{M\times [0, T]}t^{1-\sigma}|\Rc(g(t))|_{g(t)} < \infty,
\]
for some $\sigma \in (0, 1)$, there exists $T^{\prime} =T^{\prime}(n, K_0, T) > 0$ such that 
\[
\sup_{M\times[0, T^{\prime}]}|\Rm(g(t))|_{g(t)} < \infty.
\] 
\end{corollary}
\begin{proof}
 According to Shi's existence theorem \cite{Shi}, there exist positive numbers $\tilde{K} = \tilde{K}(n, K_0)$ and $\tilde{T} = \tilde{T}(n, K_0)$ and a solution $\gt(t)$ 
to \eqref{eq:rf} on $[0, \tilde{T}]$ for which $\gt(0) = g(0)$ and $\gt(t) \in \Met(M, \tilde{K})$ for each $t$. 
 As observed above, since $\sigma > 0$, the uniform bound on $t^{1-\sigma}|\Rc(g(t))|_{g(t)}$ implies
that each $g(t)$ is uniformly equivalent to $g(0)$, and is therefore also complete. We may then apply Theorem \ref{thm:uniqueness} to conclude that
 $g(t) = \gt(t)$ on $M\times [0, \min\{T, \tilde{T}\}]$, and hence that $g(t)\in \Met(M, \tilde{K})$ at least for $0 \leq t \leq T^{\prime} = \min\{T, \tilde{T}\}$. 
\end{proof}

Note that the above corollary concerns only the behavior of the solution near $t=0$.
In particular, we do not assert here that, if $g(0)\in \Met(M, K_0)$ for some $K_0$ 
and $\sup_{M\times [0, T)}|\Rc(g(t))|_{g(t)}$ is finite,
then $\sup_{M\times[0, T)}|\Rm(g(t))|_{g(t)}$ must also be finite for arbitrary $T > 0$. When $M$ is compact, this is a result of \v{S}e\v{s}um \cite{Sesum}; the assumption
on the Ricci curvature has since been weakened by a number of authors.  
When $M$ is noncompact, this statement is Theorem 1.4 in the paper \cite{MaCheng} of Ma and Cheng. 
The proof provided there (which is based on a blow-up argument), however, makes use of an
 implicit assumption that the curvature tensor remains bounded for a short time \cite{ChengPersonalCommunication}.

Finally, we note that Cabezas-Rivas and Wilking \cite{CabezasRivasWilking}  
and Giesen and Topping \cite{GiesenTopping2, GiesenTopping3} have constructed examples which demonstrate that,
when $M$ is noncompact, the maximal time of existence of a smooth complete solution $g(t)$ on $M$ may strictly exceed the smallest $T > 0$ 
such that $\sup_{M\times [0, T)}|\Rm(g(t))| = \infty$.

\section{Preliminaries}\label{sec:prelim}
Let $g$ and $\gt$ be Riemannian metrics on $M$, and $\nabla$ and $\delt$ their Levi-Civita connections.
 Define $h \dfn g- \gt$ and let $A\in C^{\infty}(T^{1}_2(M))$ be the smooth $(2, 1)$-tensor field satisfying $\nabla = \delt + A$. In local coordinates,
\begin{equation}\label{eq:a}
    A_{ij}^k = \Gamma^{k}_{ij} -\Gamt_{ij}^k = \frac{1}{2}g^{mk}\left(\delt_i g_{jm} + \delt_j g_{im} - \delt_m g_{ij}\right).
\end{equation}
Let $S_2(T^*M)$ denote the bundle of smooth symmetric $(2, 0)$-tensors on $M$, and let
\[
\delta_{\gt}: C^{\infty}(S_2(T^*M)) \to C^{\infty}(T^*M), \quad \delta_{\gt}^*:C^{\infty}(T^*M)\to C^{\infty}(S_2(T^*M)),
\]
denote, respectively, the divergence operator associated to $\gt$ and its formal $L^2(d\mu_{\gt})$-adjoint.
Explicitly,
\[
\delta_{\gt}(V)_k \dfn -\gt^{ij}\nabt_i V_{jk}, \quad \delta^*_{\gt}(W)_{ij} \dfn \frac{1}{2}\left(\nabt_i W_j + \nabt_j W_i\right).
\]
Further, define the operator $\Lc = \Lc_{g, \nabt}: C^{\infty}(S_2(T^*M))\to C^{\infty}(S_2(T^*M))$ by
\begin{align*}
    \Lc(V)_{ij} \dfn \nabt_p(g^{pq}\nabt_q V_{ij}) = g^{pq}\nabt_p\nabt_q V_{ij} - g^{pr}g^{qs}\nabt_{p}g_{rs}\nabt_q V_{ij}.
\end{align*}
\begin{remark} Here and elsewhere, $g^{ij}$ and $\gt^{ij}$ denote the components of the metrics induced by $g$ and $\gt$ on $T^*M$, i.e., the components
of the matrices satisfying $g^{ik}g_{kj} = \gt^{ik}\gt_{kj} = \delta_j^i$. Since we will be considering the metrics $g$ and $\gt$ simultaneously, 
 we will always raise and lower the indices of the components of other tensors explicitly to avoid confusion. We will use the notation $\Rm$ and $\Rmt$
for the $(3, 1)$-curvature tensors of $g$ and $\gt$, and write $\Rc$ and $\Rct$ for their Ricci tensors.
\end{remark}

For a metric $\hat{g}$, a connection $D$, and a smooth section $V \in C^{\infty}(S_2(T^*M))$, we define (following Hamilton and DeTurck)
the expression
\[
\operatorname{Bian}(\hat{g}, D, V)_k \dfn \hat{g}^{ij}\left(D_{i}V_{jk} 
- \frac{1}{2}D_{k}V_{ij}\right),
\]
where $\hat{g}^{ij} = (\hat{g}^{-1})^{ij}$. Note that
\[
    \operatorname{Bian}(g, \nabla, g) = \operatorname{Bian}(\gt, \nabt, \gt) = \operatorname{Bian}(g, \nabla, \Rc) = 
\operatorname{Bian}(\gt, \nabt, \Rct) = 0,
\]
the latter two identities being simply the contracted second Bianchi identities.
In the calculations below, we will single out the specific choice
\[
    B_k \dfn \operatorname{Bian}(g, \nabt, g)_k \dfn g_{pk}g^{ij}A^{p}_{ij} = g^{ij}\left(\nabt_{i}g_{jk} 
- \frac{1}{2}\nabt_{k}g_{ij}\right),
\]
which we will call the \emph{Bianchi one-form} of $g$ and $\gt$. As $\operatorname{Bian}(g, \nabt, g) = \operatorname{Bian}(g, \nabt, h)$, 
we may also regard $B$ as the application of a first-order operator to $h$.

We now use these operators to put the difference of the Ricci tensors of the metrics $g$ and $\gt$ 
into a convenient form.  Here and below, for given tensor fields $V$ and $W$, the notation
$V\ast W$ will represent some weighted sum of contractions of $V\otimes W$ with respect to the metric $\gt$
with coefficients bounded by universal constants.  The following computation is standard.

\begin{lemma}\label{lem:rcdiff}
The difference of the Ricci tensors of $g$ and $\gt$ satisfies
\begin{align}
\begin{split}\label{eq:rcdiff}
    -2(\Rc - \Rct) &= \Lc(h)  - 2\delta^*_{\gt}B + g^{-1}\ast g^{-1}\ast \nabt h \ast \nabt h
+ g^{-1}\ast \Rmt \ast h.
\end{split}
\end{align}
\end{lemma}
\begin{proof}
Begin with the identity
\begin{align*}
 R_{jk} - \Rt_{jk} &= \nabt_l A_{jk}^l - \nabt_jA_{kl}^l + A_{pl}^lA^p_{jk} - A_{jp}^lA^p_{kl}.
\end{align*}
Using the identity \eqref{eq:a} and that $\nabt_k g^{ij} = -g^{ip}g^{jq}\nabt_kg_{pq}$, i.e., $\nabt g^{-1} = g^{-1}\ast g^{-1}\ast \nabt g$, we have that
\begin{align*}
    -2\nabt_{l}A^l_{jk} &= \nabt_l\left(g^{ml}\left(\nabt_mg_{jk} - \nabt_jg_{km} - \nabt_kg_{jm}\right)\right)\\
    &= \mathcal{L}(g)_{jk} - g^{ml}\left(\nabt_l\nabt_j g_{km} + \nabt_l\nabt_k g_{jm}\right) + g^{-1}\ast g^{-1}\ast\nabt g \ast \nabt g\\
   &= \mathcal{L}(h)_{jk} - g^{ml}\left(\nabt_l\nabt_j h_{km} + \nabt_l\nabt_k h_{jm}\right) + g^{-1}\ast g^{-1}\ast\nabt h \ast \nabt h\\
    &= \mathcal{L}(h)_{jk} - g^{ml}\left(\nabt_j\nabt_l h_{km} + \nabt_k\nabt_l h_{jm}\right)  + g^{-1}\ast g^{-1}\ast \nabt h \ast \nabt h\\
    &\phantom{=} + g^{ml}\left(\Rt_{ljk}^ph_{pm} + \Rt_{ljm}^ph_{kp} + \Rt_{lkj}^ph_{pm} + \Rt_{lkm}^ph_{jp}\right) \\    
  &= \mathcal{L}(h)_{jk} - \nabt_j\left(g^{ml}\nabt_l g_{km}\right)  - \nabt_k\left(g^{ml}\nabt_l g_{jm}\right)  + g^{-1}\ast g^{-1}\ast \nabt h \ast \nabt h\\
    &\phantom{=} + g^{-1}\ast \Rmt \ast h.
\end{align*}
Here, $\Rt_{ljk}^p = \gt^{pq}\Rt_{ljkq}$. On the other hand, 
\begin{align*}
 2\nabt_jA^{l}_{kl} &= \nabt_j\left(g^{ml}\left(\nabt_k g_{lm} +\nabt_l g_{km} - \nabt_m g_{kl}\right)\right)\\
		    &= \nabt_j\left(g^{ml}\left(\nabt_k g_{lm}\right)\right) = \nabt_j\nabt_k\log\det(\gt^{ac}g_{cb})\\
		    &= \frac{1}{2}\nabt_j\nabt_k\log\det(\gt^{ac}g_{cb}) + \frac{1}{2}\nabt_k\nabt_j\log\det(\gt^{ac}g_{cb})\\
		    &= \frac{1}{2} \nabt_j\left(g^{ml}\left(\nabt_k g_{lm}\right)\right)
      + \frac{1}{2}\nabt_k\left(g^{ml}\left(\nabt_j g_{lm}\right)\right),
\end{align*}
while
\[	
  2\left(A_{lp}^lA^p_{jk} - A_{jp}^lA^p_{lk}\right) = g^{-1}\ast g^{-1}\ast \nabt h \ast\nabt h.
\]
Combining these three identities proves equation \eqref{eq:rcdiff}.
\end{proof}

\section{Evolution equations} We will now assume that $g = g(t)$ and $\gt = \gt(t)$ are two solutions to the Ricci flow on 
$M\times [0, T]$, and continue to use the notation $h = g -\gt$, $A = \nabla - \nabt$, and $B = \operatorname{Bian}(g, \nabt, g)$.  
The viability of our choice of energy quantity relies on the following simple computation which shows that the time derivative of $B$ 
does not depend on the full difference $\nabla \Rc - \nabt \Rct$  of the covariant derivatives of the
Ricci tensors of the solutions.

\begin{lemma} The Bianchi one-form $B = \operatorname{Bian}(g, \nabt, g)$ evolves by the equation
\begin{align}\label{eq:bev}
\begin{split}
  \pdt B_k &= - 2R_{kp}g^{pq}B_q + 2g^{ma}g^{lb}R_{ab}\left(\nabt_m h_{lk} - \frac{1}{2}\nabt_kh_{ml} \right) \\
&\phantom{=}- 2g_{pk}g^{mc}\gt^{ld}\gt^{ps}h_{cd}\left(\nabt_m\Rt_{ls}
  - \frac{1}{2}\nabt_s\Rt_{lm}\right).
\end{split}
\end{align}
\end{lemma}
\begin{proof}
To begin with, we compute that
\begin{align*}
\begin{split}
 \pdt B_k &= \pdt\left(g^{ml}\left(\nabt_m g_{kl} - \frac{1}{2}\nabt_k g_{ml}\right)\right)\\
  &= 2g^{ma}g^{lb}R_{ab}\left(\nabt_{m}g_{kl} - \frac{1}{2}\nabt_kg_{lm}\right) - 2g^{ml}\left(\nabt_m R_{kl} - \frac{1}{2}\nabt_kR_{lm} \right)\\
 &\phantom{=}
 - g^{ml}\left(\pdt\Gamt_{ml}^pg_{pk} + \pdt\Gamt_{mk}^pg_{pl} -\frac{1}{2}\pdt\Gamt_{kl}^pg_{pm} - \frac{1}{2}\pdt\Gamt^p_{km}g_{pl}\right)
\end{split}\\
  &= 2g_{pk}g^{ma}g^{lb}R_{ab}A_{ml}^p - 2g^{ml}\left(\nabt_m R_{kl} - \frac{1}{2}\nabt_kR_{lm} \right)- g^{ml}\pdt\Gamt_{ml}^pg_{pk}.
\end{align*}
Since $\operatorname{Bian}(g, \nabla, \Rc) = 0$, we can rewrite the second term using
\begin{align*}
  &g^{ml}\left(\nabt_m R_{kl} - \frac{1}{2}\nabt_kR_{lm}\right)
= g^{ml}\left(A_{mk}^pR_{pl} + A_{ml}^pR_{kp} - \frac{1}{2}A_{kl}^pR_{pm} - \frac{1}{2}A_{km}^pR_{lp}\right)\\
&\quad = g^{ml}R_{kp}A_{ml}^p = R_{kp}g^{pq}B_q.  
\end{align*}
For the last term, we use the standard formula
\[
 \pdt \Gamt_{ml}^p = \gt^{pq}\left(\nabt_{q}\Rt_{ml} - \nabt_{m}\Rt_{lq} - \nabt_{l}\Rt_{mq}\right),
\]
from which, using $\operatorname{Bian}(\gt, \nabt, \Rct) = 0$ and $g^{ml} - \gt^{ml} = -g^{mc}\gt^{ld}h_{cd}$,
we find that
\begin{align*}
 -g^{ml}\pdt\Gamt_{ml}^pg_{pk} &= 2g^{ml}\gt^{ps}g_{pk}\left(\nabt_m \Rt_{ls} - \frac{1}{2}\nabt_s \Rt_{ml}\right)\\
  &= 2(g^{ml} - \gt^{ml})\gt^{ps}g_{pk}\left(\nabt_m \Rt_{ls} - \frac{1}{2}\nabt_{s}\Rt_{lm}\right)\\
  &= -2g^{mc}\gt^{ld}\gt^{ps}g_{pk}h_{cd}\left(\nabt_m \Rt_{ls}- \frac{1}{2}\nabt_s\Rt_{lm}\right).
\end{align*}
Combining the above three identities proves \eqref{eq:bev}.
\end{proof}

From \eqref{eq:rcdiff} and \eqref{eq:bev}, we obtain the following expressions for the evolutions of the norms of $h$ and $B$
relative to the metric $\gt$.

\begin{lemma}
Let $g$ and $\gt$ be solutions to the Ricci flow on $M\times [0, T]$ and
denote by $\langle \cdot, \cdot\rangle \dfn \langle \cdot, \cdot\rangle_{\gt}$ and $|\cdot| \dfn |\cdot|_{\gt}$
the inner products and norms induced by $\gt$ on the bundles $T^k_l(M)$. Then there is a constant $C_0 = C_0(n)$
such that
\begin{align}
\label{eq:hnormev}
  \pdt |h|^2 &\leq 2 \langle \Lc(h) -2\delta^*_{\gt}B, h\rangle + C_0(|\Rct| 
+ |g^{-1}||\Rmt|)|h|^2 + C_0|g^{-1}|^2|h||\nabt h|^2\\
\begin{split}
\label{eq:bnormev}
  \pdt |B|^2 &\leq C_0\left(|\Rct| + |g^{-1}||\Rc|\right)|B|^2 + C_0|g||g^{-1}||\nabt\Rct||h||B|  \\
  &\phantom{\leq}+ C_0|g^{-1}|^2|\Rc||\nabt h||B|.
\end{split}
\end{align}
\end{lemma}
In the above equations, the curvature of $g$ only appears explicitly through the factors of $|\Rc|$ in the evolution equation for $|B|^2$. 
As we mentioned above, this is one advantage of the present method over that in \cite{KotschwarRFUniqueness}.
\section{The energy argument}

We now set about to prove Theorem \ref{thm:uniqueness}. First we modify a lemma from \cite{KotschwarRFUniqueness}
which follows a construction of \cite{KarpLi} (see also Theorem 12.22 of \cite{RFV2P2}). 
\begin{lemma}\label{lem:cutoffgrowth}
Suppose that $\bar{g}$ is a complete metric on $M$ and $\gt(t)$ is a smooth family of complete metrics on $M\times [0, T]$ satisfying $\bar{g} \leq \gamma \gt(t)$
for some fixed $\gamma > 0$.
Define $\bar{r}(x) \dfn \operatorname{dist}_{\bar{g}}(x, x_0)$ for some $x_0\in M$. 
  Then, for any $L_1$, $L_2 > 0$, there exists a positive constant $T^{\prime} = T^{\prime}(\gamma, n, L_1, L_2)$
and, for each $0 < \tau \leq T^{\prime}$, a smooth function $\eta = \eta_{\tau}: M\times [0, \tau]\to (0, \infty)$ such
that
\[
-\pd{\eta}{t} + L_1 |\nabt\eta|^2_{\gt(t)} \leq 0,\quad\mbox{and}\quad
   \eta(x, t) \geq L_2\bar{r}^2(x),
\]
 on $M\times[0, \min\{\tau, T\}]$.
\end{lemma}
\begin{proof}
According to Proposition 2.1 of \cite{GreeneWu}, we can find $\rho\in C^{\infty}(M)$ satisfying that
$\bar{r}(x) \leq \rho(x) \leq \bar{r}(x) + 1$ and  $|\delb \rho|_{\bar{g}} \leq 2$.
Then, as in Lemma 5 of \cite{KotschwarRFUniqueness}, if one successively chooses $0 < \beta \leq 1/(4L_1\gamma)$ and
$T^{\prime} =\sqrt{\beta/(8L_2)}$, the function $\eta_{\tau}(x, t) = \beta\rho^2(x)/(4(2\tau - t))$
will satisfy the desired conditions for any $0 < \tau \leq \min\{T,  T^{\prime}\}$.
\end{proof}
 
Next we update inequalities \eqref{eq:hnormev} and \eqref{eq:bnormev} to reflect the assumptions of Theorem \ref{thm:uniqueness}. In the rest of this section, we will write $T^* = \max\{T, 1\}$,
and use the unadorned notation $\langle \cdot, \cdot\rangle = \langle \cdot, \cdot \rangle_{\gt(t)}$ and $|\cdot| = |\cdot|_{\gt(t)}$ to represent the
inner products and norms induced by $\gt(t)$ on the various tensor bundles $T^k_l(M)$. 
\begin{lemma}\label{lem:hbnormest}
 Under the assumptions of Theorem \ref{thm:uniqueness}, there exists a positive constant $N_0$, depending only
 on $\sigma$, $n$, $K$, and $T^*$, such that 
\begin{align}
\label{eq:hnormest}
 \pdt \left(\frac{|h|^2}{t^{1+\sigma}}\right) & \leq -\frac{1+\sigma}{t^{2+\sigma}}(1- N_0t^{\sigma})|h|^2 + 
 \frac{2}{t^{1+\sigma}} \langle \Lc(h) -2\delta^*_{\gt}B, h\rangle + \frac{N_0}{t}|\nabt h|^2,\\
 \label{eq:bnormest}
  \pdt\left(\frac{|B|^2}{t^{\sigma}}\right)&\leq -\frac{\sigma}{2t^{1+\sigma}}(1- N_0t^{\sigma})|B|^2 + \frac{N_0}{t^{2-\sigma}}|h|^2 + \frac{N_0}{t}|\nabt h|^2 
\end{align} 
on $M\times (0, T]$.
\end{lemma}
\begin{proof}
We will use $N$ to denote a sequence of positive constants depending on $\sigma$, $n$, $K$, and $T^*$. 
We note first that, from the curvature bounds \eqref{eq:curvgrowth}, it follows (see, e.g., Lemma 6.49 of \cite{ChowKnopf}) that
the families of metrics $g(t)$ and $\gt(t)$ are mutually uniformly equivalent, and in fact,
\[
    N^{-1}\gt(s) \leq g(t) \leq N\gt(s), \quad N^{-1}\gt(s)\leq \gt(t) \leq N\gt(s),
\]
for any $s$, $t\in [0, T]$. In particular, we have uniform bounds of the form
\[
  |g^{-1}| + |g|+ t^{1-\sigma}|\Rc| \leq N, \quad |h|\leq Nt^{\sigma}, 
\]
on $M\times [0, T]$. (The point of the first bound is that the terms on the left involve the norm $|\cdot| = |\cdot|_{\gt}$ rather than $|\cdot|_{g}$;
the second bound can be deduced from a simple pointwise estimate as in, e.g., Lemma 6 of \cite{KotschwarRFUniqueness}.) Also, by Shi's estimates \cite{Shi} (see Corollary 9 in \cite{KotschwarRFUniqueness}), the assumptions on $\gt$ imply a bound
of the form $t^{3/2 - \sigma}|\nabt \Rct| \leq N$ on $M\times [0, T]$. 

Combined with these observations, inequalities \eqref{eq:hnormev} and \eqref{eq:bnormev} say that
\begin{align*}
   \pdt |h|^2 &\leq 2\langle \Lc(h) -2\delta^*_{\gt}B, h\rangle + \frac{N}{t^{1-\sigma}}|h|^2 + Nt^{\sigma}|\nabt h|^2,\\
\begin{split}
  \pdt |B|^2 &\leq \frac{N}{t^{1-\sigma}}\left(|B|^2 + |\nabt h||B| + \frac{|h||B|}{t^{1/2}}\right),
\end{split}
\end{align*}
on $M\times (0, T]$. Thus,
\begin{align}
\label{eq:hnormest2}
\begin{split}
   \pdt \left(\frac{|h|^2}{t^{1+\sigma}}\right) &\leq -\frac{1+\sigma}{t^{2+\sigma}}(1-Nt^{\sigma})|h|^2 
   + \frac{2}{t^{1+\sigma}}\langle \Lc(h) -2\delta^*_{\gt}B, h\rangle
      + \frac{N}{t}|\nabt h|^2,
\end{split}\\
\begin{split}
\label{eq:bnormest2}
  \pdt \left(\frac{|B|^2}{t^{\sigma}}\right) &\leq -\frac{\sigma}{t^{1+\sigma}}(1-Nt^{\sigma})|B|^2 +  \frac{N}{t}|\nabt h||B| + \frac{N}{t^{3/2}}|h||B|,
\end{split}
\end{align}
on $M\times (0, T]$.  Estimating the last two terms in \eqref{eq:bnormest2} by
\[
\frac{N}{t}|\nabt h||B| \leq \frac{N}{t}\left(|\nabt h|^2 + |B|^2\right),
\quad \frac{N}{t^{3/2}}|h||B| \leq  \frac{N}{t^{2-\sigma}}|h|^2 + \frac{\sigma}{2t^{1+\sigma}}|B|^2,
\]
we obtain \eqref{eq:hnormest} and \eqref{eq:bnormest}.
\end{proof}
 
\subsection{Proof of Theorem \ref{thm:uniqueness}} 
Now we prove the main result of this paper. At this point, the argument is essentially a special
case of that of Theorem 13 in \cite{KotschwarRFUniqueness}, but, since the statement of that theorem does not exactly match our situation, 
and, in any case, our argument is short, we will give the details here. We discuss the relationship of the argument here
to the formulation in \cite{KotschwarRFUniqueness} somewhat further in Section \ref{ssec:pdeode} below.

\begin{proof}[Proof of Theorem \ref{thm:uniqueness}] 
We will give the proof for noncompact $M$; the argument for compact $M$ follows the same lines but is much simpler since we have uniform bounds for the curvature tensors
of both solutions (and their covariant derivatives) and can work instead with a global energy quantity.
To reduce clutter, we will use $C$
to denote a sequence of positive constants depending only on $n$ and, as before, use $N$ to denote a sequence of positive constants
depending on $\sigma$, $n$, $K$, and $T^*$.

We will need a cut-off function. Fix $x_0\in M$ and define $\bar{r}(x) \dfn d_{\bar{g}}(x_0, x)$ where $\bar{g}= \gt(T)$.
Again using \cite{GreeneWu}, we can find $\rho\in C^{\infty}(M)$ satisfying $\bar{r}(x) \leq \rho(x) \leq \bar{r}(x) + 1$
and $|\delb \rho|_{\bar{g}} \leq 2$. Let $\phi\in C_c^{\infty}(\RR, [0, 1])$ be a monotonically decreasing function
which is identically one on  $(-\infty, 2]$, is supported in $(-\infty, 3)$, and satisfies
$(\phi^{\prime})^2\leq 10 \phi$.  For $r > 0$, the function $\theta_r(x) \dfn \phi(\rho(x)/r)$ will then belong to $C^{\infty}_c(B_{\bar{g}}(x_0, 3r))$, and satisfy $\theta_r \equiv 1$ on $B_{\bar{g}}(x_0, r)$ and 
$|\bar{\nabla}\theta_r|_{\bar{g}}^2 \leq 40r^{-2}\theta_r$ on all of $M$.

For the time-being, we fix some $0 < \tau \leq T$ and let 
$\eta:M\times [0, \tau]\to \RR$ denote an arbitrary smooth positive function that is increasing in $t$.
We will later use Lemma \ref{lem:cutoffgrowth} to specify $\eta = \eta_{\tau}$ with $\tau = \min\{T, T^{\prime}\}$
for some sufficiently small $T^{\prime} > 0$ depending on $\sigma$, $n$, and $K$, and $T^*$.

For each $r > 0$
 and $t\in [0, \tau]$ we define
\[
   \Bc_r(t) \dfn \int_M|B|^2_{\gt(t)}\theta_r e^{-\eta}\,d\mu_{\gt(t)}, \quad \Hc_r(t) \dfn \int_M|h|^2_{\gt(t)}\theta_r e^{-\eta}\,d\mu_{\gt(t)},
\]
and
\[
   \Kc_r(t) \dfn \int_M|\nabt h|^2_{\gt(t)}\theta_r e^{-\eta}\,d\mu_{\gt(t)}.
\]
(When $M$ is compact, we can take $\theta_r \equiv 1$ and $\eta \equiv 0$.)
The bulk of the argument consists of deriving an appropriate differential inequality for a weighted combination of $\Bc_r(t)$ and $\Hc_r(t)$.

We start by considering the derivative of $t^{-\sigma}\Bc_r(t)$. Fix $r > 0$ and note that, by \eqref{eq:bnormest} and our assumptions that $t^{1-\sigma}|\Rt|\leq N$ and $\pd{\eta}{t}\geq 0$, we have
\begin{align}\nonumber
\frac{d}{dt}\left(\frac{\Bc_{r}(t)}{t^{\sigma}}\right) &= \int_M 
\left(\pdt\left(\frac{|B|^2}{t^{\sigma}}\right) - \frac{|B|^2}{t^{\sigma}}\left(\Rt +\pd{\eta}{t}\right)\right)\theta_r e^{-\eta}\,d\mu_{\gt}\\
\nonumber
&\leq\int_M\left(\frac{N}{t^{2-\sigma}}|h|^2 + \frac{N}{t}|\nabt h|^2 -\frac{\sigma}{2t^{1+\sigma}}(1-Nt^{\sigma})|B|^2\right)\theta_r e^{-\eta}\,d\mu_{\gt}\\
\label{eq:brev}
&\leq -\frac{\sigma}{2t^{1+\sigma}}(1 - Nt^{\sigma})\Bc_r(t)
    + \frac{N}{t^{2-\sigma}}\Hc_r(t) + \frac{N}{t}\Kc_r(t)
\end{align}
on $[0, \tau]$.

Next we consider the derivative of $t^{-(1+\sigma)}\Hc_r(t)$. The uniform equivalence of $g(t)$ and $\gt(t)$ implies that we have
a lower bound of the form
\[
\alpha_0 |\nabt h|^2\leq g^{ij}\langle \nabt_i h, \nabt_j h\rangle
\] 
on $M\times [0, T]$ for some positive constant $\alpha_0 = \alpha_0(\sigma, n, K, T^*)$, so, integrating by parts and 
using that $|B| \leq C|g^{-1}||\nabt h|$ and $|\delta_{\gt}h| \leq C|\nabt h|$, we obtain the inequality
\begin{align*}
\begin{split}
& \int_M\langle \Lc(h) -2\delta^*_{\gt}B, h\rangle\theta_r e^{-\eta}\,d\mu_{\gt}\\
&= -\int_{M}\left(g^{ij}\langle \nabt_i h, \nabt_j h\rangle +2\langle \delta_{\gt}h, B\rangle
	-g^{ij}\nabt_i\eta\langle\nabt_j h, h\rangle + 2h(\nabt \eta, B^{\sharp})\right)\theta_r e^{-\eta}\,d\mu_{\gt}\\
&\phantom{=} -\int_M\left( g^{ij}\nabt_i\theta_r \langle\nabt_j h, h\rangle 
- 2h(\nabt \theta_r, B^{\sharp})\right)e^{-\eta}\,d\mu_{\gt}
\end{split}\\
\begin{split}
&\leq -\alpha_0\Kc_r + C_1\int_M\left(|\nabt h||B| + |\nabt \eta|\left(|g^{-1}||\nabt h||h| + |h||B|\right)\right)\theta_r e^{-\eta}\,d\mu_{\gt}\\
&\phantom{\leq} + C_1\int_M|\nabt\theta_r|\left(|g^{-1}||\nabt h||h| + |h||B|\right)e^{-\eta}\,d\mu_{\gt}
\end{split}
\end{align*}
for some constant $C_1 = C_1(n)$.
Here $(B^{\sharp})^i = \gt^{ij}B_j$. 

We now estimate the two integrals on the right. We have 
\[
  |\nabt h||B| \leq \frac{\alpha_0}{6C_1}|\nabt h|^2 + N|B|^2,
\]
and using again that $|g^{-1}|\leq N$, also that
\[
  |\nabt \eta|(|g^{-1}||h||\nabt h| + |h||B|) \leq N|\nabt \eta|^2|h|^2+\frac{\alpha_0}{6C_1}|\nabt h|^2 + |B|^2. 
\]
Further, on the set
\[
U_r \dfn \{\,x\in M\,|\, \nabt\theta_r \neq 0\,\} \subset B_{\bar{g}}(x_0, 3r) \setminus B_{\bar{g}}(x_0, r)
\]
we have
\[
 |\nabt\theta_r|\left(|g^{-1}||\nabt h||h| + |B||h|\right) \leq \frac{\alpha_0}{6C_1}|\nabt h|^2\theta_r +|B|^2\theta_r  + \frac{N|\nabt\theta_r|^2}{\theta_r}|h|^2,
\]
so the above inequality becomes
\begin{align}
\begin{split}\label{eq:ppart}
& \int_M\langle \Lc(h) -2\delta^*_{\gt}B, h\rangle\theta_r e^{-\eta}\,d\mu_{\gt}\\
&\quad\leq -\frac{\alpha_0}{2}\Kc_r(t) + N\Bc_r(t) + N\int_M\left(|\nabt\eta|^2\theta_r 
+ \theta_r^{-1}|\nabt \theta_r|^2\chi_{U_r}\right)|h|^2e^{-\eta}\,d\mu_{\gt}.
\end{split}
\end{align}

Starting from \eqref{eq:hnormest}, then, we compute that
\begin{align*}
\begin{split}
\frac{d}{dt}\left(\frac{\Hc_r(t)}{t^{1+\sigma}}\right)
&= \int_M\left(\pdt\left(\frac{|h|^2}{t^{1+\sigma}}\right) 
-\frac{1}{t^{1+\sigma}}\left(\tilde{R} + \pd{\eta}{t}\right)|h|^2\right)\theta_re^{-\eta}\,d\mu_{\gt}\\
&\leq\int_M\bigg(\frac{2}{t^{1+\sigma}}\langle \Lc(h) -2\delta^*_{\gt}B, h\rangle +\frac{N}{t}|\nabt h|^2\\
&\phantom{\leq\int_M\bigg(}-\left(\frac{1+\sigma}{t^{2+\sigma}}\right)(1-Nt^{\sigma})|h|^2
-\frac{1}{t^{1+\sigma}}\pd{\eta}{t}|h|^2\bigg)\theta_re^{-\eta}\,d\mu_{\gt},
\end{split}
\end{align*}
which, with \eqref{eq:ppart}, implies that
\begin{align}\label{eq:hrev}
\begin{split}
\frac{d}{dt}\left(\frac{\Hc_r(t)}{t^{1+\sigma}}\right) &\leq - \frac{1+\sigma}{t^{2+\sigma}}\left(1 - N_1t^{\sigma}\right)\Hc_r(t)
       - \frac{\alpha_0}{t^{1+\sigma}}\left(1-N_1t^{\sigma}\right)\Kc_r(t)\\
	&\phantom{\leq}+ \frac{N_1}{t^{1+\sigma}}\Bc_r(t) +\frac{1}{t^{1+\sigma}}\int_M\left(N_1|\nabt \eta|^2 - \pd{\eta}{t}\right)|h|^2\theta_r e^{-\eta}\,d\mu_{\gt}\\
	&\phantom{\leq}+ \frac{N_1}{t^{1+\sigma}}\int_{U_r}\!\!\!\theta_r^{-1}|\nabt\theta_r|^2|h|^2 e^{-\eta}\,d\mu_{\gt},
\end{split}
 \end{align}
for some $N_1 = N_1(\sigma, n, K, T^{*}) > 0$ on $[0, \tau]$.

We now specify $\eta$. Since $\bar{g}\leq N\gt(t)$, Lemma \ref{lem:cutoffgrowth} implies that there is 
$T^{\prime} = T^{\prime}(\sigma, n, K, T^*) > 0$ such that for $\tau \in (0, \min\{T, T^{\prime}\}]$, 
the function $\eta=\eta_{\tau}$ satisfies
\[
   N_1|\nabt\eta|^2 -\pd{\eta}{t}\leq 0, \quad \eta(x, t) \geq \bar{r}^2(x),
\]
on $M\times[0, \tau]$. We will henceforth define $\tau \dfn \min\{T^{\prime}, T\}$ and $\eta \dfn \eta_{\tau}$.
With this choice, the penultimate integral in \eqref{eq:hrev}
is nonpositive and $\eta(x, t) \geq r^2$ on $U_r \times [0, \tau]$. 

For the last integral in \eqref{eq:hrev}, note that, since the Ricci curvatures of $\bar{g} = \gt(T)$ have a lower bound of the form $-CKT^{\sigma-1} = NT^{-1}$, the Bishop-Gromov inequality
and the uniform equivalence of the metrics $\gt(t)$ and $\bar{g}$ imply an estimate of the form
\begin{equation}\label{eq:volest}
      \operatorname{vol}_{\gt(t)}(B_{\bar{g}}(x_0, r)) \leq N\operatorname{vol}_{\bar{g}}(B_{\bar{g}}(x_0, r)) \leq Ne^{\bar{V}r}
\end{equation}
for all $r$ for some $\bar{V} = \bar{V}(\sigma, n, K, T)$.  Also, by our choice of $\theta_r$ and the uniform equivalence of $\gt(t)$ and $\bar{g}$, 
\[
|\nabt\theta_r|^2 \leq N |\delb \theta_r|^2_{\bar{g}} \leq N\theta_r/r^2
\]
on $M\times[0, \tau]$. Using that $U_r \subset B_{\bar{g}}(x_0, 3r)$, and that (as observed in Lemma \ref{lem:hbnormest}) we have $|h|^2 \leq Nt^{2\sigma}$ on $M\times [0, T]$,
we then obtain the estimate
\[
\frac{1}{t^{1+\sigma}}\int_{U_r}\theta_r^{-1}|\nabt\theta_r|^2|h|^2 e^{-\eta}\,d\mu_{\gt}
  \leq \frac{N}{t^{1-\sigma}r^2}e^{-r^2+\bar{V}r}
\]
for $t\in (0,\tau]$.

Incorporating these observations into \eqref{eq:hrev} and recalling \eqref{eq:brev}, we arrive at the system
of inequalities
\begin{align*}
 \begin{split}
\frac{d}{dt}\left(\frac{\Bc_{r}(t)}{t^{\sigma}}\right) &\leq -\frac{\sigma}{2t^{1+\sigma}}(1 - N_2t^{\sigma})\Bc_r(t)
    + \frac{N_2}{t^{2-\sigma}}\Hc_r(t) + \frac{N_2}{t}\Kc_r(t)\\
\frac{d}{dt}\left(\frac{\Hc_r(t)}{t^{1+\sigma}}\right)
 &\leq \frac{N_2}{t^{1+\sigma}}\Bc_r(t) - \frac{1+\sigma}{t^{2+\sigma}}\left(1 - N_2t^{\sigma}\right)\Hc_r(t)
       - \frac{\alpha_0}{t^{1+\sigma}}\left(1-N_2t^{\sigma}\right)\Kc_r(t)\\
&\phantom{\leq}
 +  \frac{N_2}{t^{1-\sigma}}e^{-r^2+\bar{V}r},
\end{split}
\end{align*}
valid for all $t \leq \tau$ and $r \geq 1$, for some fixed $N_2 = N_2(\sigma, n, K, T^*) > 0$.
 
 Choosing $a = a(\sigma, n, K, T^*)$ such that $a\sigma/4 >  N_2$, we at last obtain that, for all $r \geq 1$,
 the quantity
\[
    \Ec_r(t) \dfn a\frac{\Bc_r(t)}{t^\sigma} + \frac{\Hc_r(t)}{t^{1+\sigma}} 
\]
satisfies
\begin{align}\label{eq:erev}
\begin{split}
  \frac{d}{dt}\Ec_r(t) &\leq -\frac{a\sigma}{4t^{1+\sigma}}(1 - N_3t^{\sigma})\Bc_r(t) 
- \frac{(1+\sigma)}{t^{2+\sigma}}(1 - N_3t^{\sigma})\Hc_r(t)\\
  &\phantom{\leq} - \frac{\alpha_0}{t^{1+\sigma}}(1 - N_3t^{\sigma})\Kc_r(t) + \frac{N_3}{t^{1-\sigma}}e^{-r^2+ \bar{V}r}      
\end{split}
\end{align}
on $(0, \tau]$ for some $N_3 = N_3(\sigma, n, K, T^*)$.  

From this we see that if we choose $T^{\prime\prime} = \min\{ N_3^{-1/\sigma}, \tau\}$,
the coefficients of the first three terms
in \eqref{eq:erev} will be nonpositive for all $t\in (0, T^{\prime\prime}]$. That is, we have
\begin{align*}
\begin{split}
 \frac{d}{dt}\Ec_r(t) &\leq \frac{N_3}{t^{1-\sigma}}e^{-r^2+ \bar{V}r} 
\end{split}
\end{align*}
for all $r \geq 1$ on $(0, T^{\prime\prime}]$.  Fixing $0 < t_0 <  t_1 \leq T^{\prime\prime}$ and $r \geq 1$ and integrating the last inequality over $[t_0, t_1]$ yields
\[
    \Ec_r(t_1) - \Ec_r(t_0) \leq N(t_1^{\sigma} - t_0^{\sigma})e^{-r^2 + \bar{V}r}.
\]
Since $\theta_r$ has compact support, for any $b > 0$, we have
\[
    \sup_{(\operatorname{supp} \theta_r)\times (0, T]}\frac{1}{t^b}(|h| + |B|) < \infty.
\]
It follows that $\lim_{t_0\searrow 0} \Ec_r(t_0) = 0$ and
\[
 \Ec_r(t_1) \leq Ne^{-r^2 + \bar{V}r}
\]
for all $t_1 \in (0, T^{\prime\prime}]$ and all $r \geq 1$.
Letting $r$ tend to infinity, we conclude that $h(t_1)\equiv 0$ on $M$, and, since $t_1$ was arbitrary, that $h \equiv 0$ on all of $M\times [0, T^{\prime\prime}]$. Now, by definition, 
$T^{\prime\prime} = \min\{N_3^{-1/\sigma}, T^{\prime}, T\}$ where $N_3$ and $T^{\prime}$ depend only on the fixed parameters
$\sigma$, $n$, $K$, and $T^*$. Thus if $T^{\prime\prime} < T$, we can iterate this argument to conclude that $h\equiv 0$ on all of $M\times [0, T]$. 
\end{proof}

\subsection{A reformulation}\label{ssec:pdeode} The argument above can be recast in terms of
the general formulation of  Section 4 of \cite{KotschwarRFUniqueness}. For example, using the notation of that section, we can define the tensors
\[
  X \dfn  t^{-(1+\sigma)/2}h, \quad Y \dfn  at^{-\sigma/2}B
\]
and
\[
  U_{ij}^k \dfn t^{-(1+\sigma)/2}((g^{kp} - \gt^{kp})\nabt_{p}h_{ij} - \delta_i^kB_{j} - \delta_j^kB_i),
\]
on $M\times (0, T]$ where, as in the proof in the previous section, $a > 0$ is some parameter we may choose to be arbitrarily large.
From \eqref{eq:hnormest} and \eqref{eq:bnormest}, it is not hard to see that, under the hypotheses of Theorem \ref{thm:uniqueness}, $X$, $Y$, and $U$ satisfy a system of the form
\begin{align*}
\begin{split}
  \left\langle\pd{X}{t} - \widetilde{\Delta}X - \operatorname{div}_{\gt} U, X \right\rangle &\leq -\frac{1+\sigma}{2t}(1 - Nt^{\sigma})|X|^2 + \frac{N}{a^2t}|Y|^2 
  +  Nt^{\sigma}|\nabt X|^2,
\end{split}\\
\begin{split}
 \left\langle\pd{Y}{t}, Y\right\rangle & \leq  \frac{Na^2}{t^{1-2\sigma}}|X|^2 -\frac{\sigma}{4t}(1-Nt^{\sigma})|Y|^2  + Na^2t^{\sigma}|\nabt X|^2,
  \end{split}\\
  |U| &\leq Nt^{\sigma}|\nabt X| + \frac{ N}{at^{1/2}}|Y|,
\end{align*}
for some $N = N(\sigma, n, K, T^*)$. Here $(\operatorname{div}_{\gt} U)_{ij} \dfn \nabt_kU^k_{ij}$. 
The structure of the above system does not quite match the hypotheses of Theorem 13 of \cite{KotschwarRFUniqueness} as it is stated,
but the proof given there can still be made to apply with some minor adaptations; this is essentially the proof given in the previous section.
 
Note that the ``parabolic'' component $X$ of the above system
is of order zero in the difference of the metrics while the ``ordinary-differential'' component $Y$ is of order one. 
This is the reverse of the situation in \cite{KotschwarRFUniqueness} and \cite{KotschwarKCFUniqueness}, where
the parabolic components are higher-order, constructed out of the differences of the curvature tensors and their derivatives,
and the ordinary-differential components are lower-order, constructed out of the differences of the metrics and their connections.

\begin{acknowledgement*} The author thanks Ben Chow, Lei Ni, Miles Simon, and Peter Topping for their comments and for sharing some of their intuition with him.
\end{acknowledgement*}

\end{document}